\documentclass[12pt,a4paper]{article}
\usepackage{microtype}
\usepackage[margin=1in]{geometry}
\usepackage{libertine}
\usepackage[T1]{fontenc}
\usepackage[utf8]{inputenc}
\usepackage{amsfonts,amssymb,amsmath,amsthm}
\usepackage{calc}
\usepackage{tikz}
\usetikzlibrary{calc}

\usepackage{url}
\usepackage{hyperref}
\usepackage{enumitem}
\usepackage{xspace}
\usepackage{todonotes}
\usepackage{thmtools}
\makeatletter
\@ifundefined{newcounteralias}{}{%
  \renewcommand\thmt@autorefsetup{%
    \@xa\def\csname\thmt@envname autorefname\@xa\endcsname
      \@xa{\thmt@thmname}%
  }%
}
\makeatother
\usepackage{mathrsfs}  
\usepackage{textpos}
\usepackage[normalem]{ulem}
\usepackage[capitalize, nameinlink, noabbrev]{cleveref}

\DeclareMathOperator{\tw}{tw}
\DeclareMathOperator{\bsn}{bsn}

\newtheorem{theorem}{Theorem}[section]

\newtheorem{lemma}[theorem]{Lemma}
\newtheorem{claim}{Claim}[theorem]

\theoremstyle{definition}

\newcommand{\Oh}{\mathcal{O}}

\crefname{claim}{Claim}{Claims}

\def\R{\mathbb{R}}

\def\N{\mathbb{N}}

\def\B{\mathrm{Ball}}

\newcommand{\Aa}{\mathcal{A}}
\newcommand{\Tt}{\mathcal{T}}
\newcommand{\Bb}{\mathcal{B}}

\newcommand{\Dd}{\mathcal{D}}

\newcommand{\wh}[1]{\widehat{#1}}

\newcommand{\bag}{\mathsf{bag}}
\newcommand{\dist}{\mathrm{dist}}
\newcommand{\rad}{\mathrm{rad}}

\newcommand{\ceil}[1]{\left\lceil #1 \right\rceil }

\renewcommand{\leq}{\leqslant}
\renewcommand{\geq}{\geqslant}

\renewcommand{\mid}{~|~}

\usepackage{framed}
\usepackage{tabularx}

\newlength{\RoundedBoxWidth}
\newsavebox{\GrayRoundedBox}
\newenvironment{GrayBox}[1]%
   {\setlength{\RoundedBoxWidth}{.93\columnwidth}
    \def\boxheading{#1}
    \begin{lrbox}{\GrayRoundedBox}
       \begin{minipage}{\RoundedBoxWidth}}%
   {   \end{minipage}
    \end{lrbox}
    \begin{center}
    \begin{tikzpicture}%
       \node(Text)[draw=black!20,fill=white,rounded corners,inner sep=2ex,text width=\RoundedBoxWidth]
             {\usebox{\GrayRoundedBox}};
        \coordinate(x) at (current bounding box.north west);
        \node [draw=white,rectangle,inner sep=3pt,anchor=north west,fill=white]
        at ($(x)+(6pt,.75em)$) {\boxheading};
    \end{tikzpicture}
    \end{center}}

\newenvironment{defproblemx}[1]{\noindent\ignorespaces%
                                \FrameSep=6pt%
                                \parindent=0pt%
                \begin{GrayBox}{#1}%
                \begin{tabular*}{\columnwidth}{!{\extracolsep{\fill}}@{\hspace{.1em}} >{\itshape} p{1.5cm} p{0.86\columnwidth} @{}}%
            }{
                \end{tabular*}%
                \end{GrayBox}%
                \ignorespacesafterend
            }

\usepackage{mathtools}

\newenvironment{claimproof}[1][Proof of Claim.]{%
  \begin{proof}[#1]%
}{%
  \end{proof}%
}

\newcommand{\disclaimer}{
\let\thefootnote\relax\footnotetext{\phantom{dupsko}\\ \noindent Views and opinions expressed are those of the authors only and do not necessarily reflect those of the European Union or the European Research Council Executive Agency. Neither the European Union nor the granting authority can be held responsible for them.}}

\title{On coarse tree decompositions\\and coarse balanced separators}

\author{Tara Abrishami\thanks{Stanford University (\texttt{tara.abrishami@stanford.edu}). Supported by the National Science Foundation Award Number DMS-2303251 and the Alexander von Humboldt Foundation.}
\and Jadwiga Czyżewska\thanks{University of Warsaw, Poland (\texttt{j.czyzewska@mimuw.edu.pl}).
Supported by the Polish National Science Centre SONATA BIS-12 grant number 2022/46/E/ST6/00143.}
\and Kacper Kluk\thanks{University of Warsaw, Poland (\texttt{k.kluk@mimuw.edu.pl}).
Supported by the Polish National Science Centre SONATA BIS-12 grant number 2022/46/E/ST6/00143.}
\and Marcin Pilipczuk\thanks{University of Warsaw, Poland (\texttt{m.pilipczuk@mimuw.edu.pl}).
Supported by the Polish National Science Centre SONATA BIS-12 grant number 2022/46/E/ST6/00143.}
\and Michał Pilipczuk\thanks{Technical University of Darmstadt, Germany  (\texttt{pilipczuk@mathematik.tu-darmstadt.de}). This work was partially done while Mi. Pilipczuk was affiliated with the University of Warsaw and was supported by the European Research Council (ERC) under the European Union’s Horizon 2020 research and innovation programme, grant agreement no. 948057~(BOBR). Afterwards, while employed at TU Darmstadt, Mi. Pilipczuk was supported by the Alexander von Humboldt Foundation in the
  framework of the Alexander von Humboldt Professorship, endowed by the Federal Ministry of Research, Technology, and Space, as well as by the European Union through ERC project WYDRA, grant agreement no. 101230549.}
\and Paweł Rzążewski\thanks{Warsaw University of Technology, Poland (\texttt{pawel.rzazewski@pw.edu.pl}). Supported by the European Research Council (ERC) under the European Union’s Horizon 2020 research and innovation programme, grant agreement no. 948057~(BOBR).}}

\begin{document}
\begin{titlepage}
\date{}
\maketitle

\begin{abstract}
It is known that there is a linear dependence between the treewidth of a graph and its {\em{balanced separator number}}: the smallest integer $k$ such that for every nonnegative weight function on the vertices, the graph admits a balanced separator of size at most $k$. We prove a coarse analogue of this statement: If every nonnegative weight function on the vertices of a graph admits a balanced separator coverable by $k$ balls of radius $r$,
then the graph admits a tree decomposition whose every bag can be covered by $\mathcal{O}(k^2\log k)$ balls of radius $\mathcal{O}(rk^2\log k)$ each.
\renewcommand\footnotemark{}\disclaimer
\end{abstract}

\def\thepage{}
\thispagestyle{empty}
 \begin{textblock}{20}(-1.9, 2.3)
  \includegraphics[width=40px]{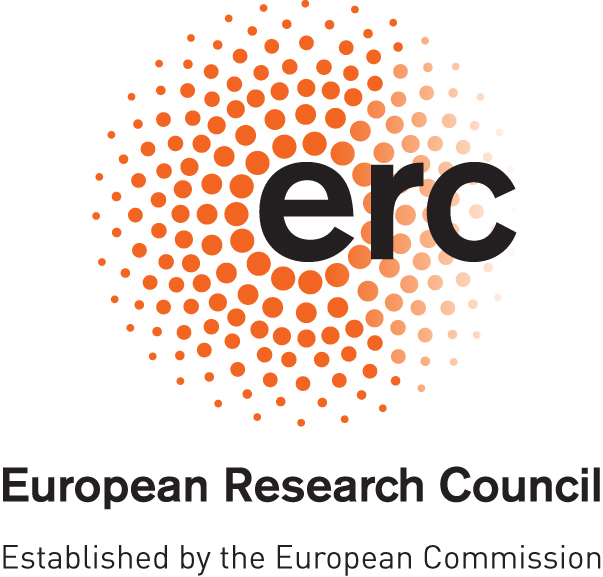}%
 \end{textblock}
 \begin{textblock}{20}(-1.9, 3.2)
  \includegraphics[width=40px]{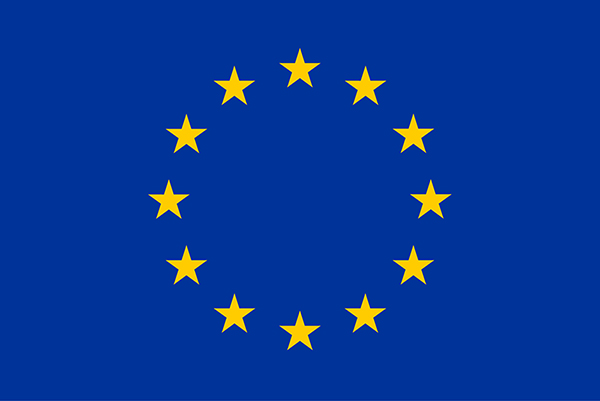}%
 \end{textblock}
\end{titlepage}

\section{Introduction}\label{sec:intro}

{\em{Coarse graph theory}} is a young branch of graph theory whose aim is to study the structure in graphs understood as metric spaces. Foundations of this research direction were laid by Georgakopoulos and Papasoglu~\cite{coarse2023}, who formulated a program of understanding coarse counterparts of the fundamental tools, techniques, and results from standard structural graph theory, particularly the theory of Graph Minors. In the analogy between the standard and the coarse settings, it is typical that the requirement of disjointness of objects is replaced with {\em{farness}}, and the requirement of intersection is replaced with {\em{closeness}}.

The goal of this paper is to prove a coarse counterpart of the classic equivalence between treewidth and balanced separators. Recall that classically, a tree decomposition of a graph $G$ has \emph{width} at most $k$ if every bag consists of at most $k+1$ vertices; and the treewidth of $G$ is the smallest possible width of a tree decomposition. A natural coarse analogue of this condition is the following: every bag can be covered by $k$ balls of radius $r$ in $G$, for some distance parameter $r\in \N$ fixed beforehand. (We will call such vertex sets {\em{$(k,r)$-coverable}}.) This coarse analogue of treewidth was proposed by Nguyen, Scott, and Seymour~\cite{coarsetw2025} and by Hickingbotham~\cite{hickingbotham2025twquasiisom}, who independently showed that graphs admitting such tree decompositions are quasi-isometric\footnote{A {\em{quasi-isometry}} is a mapping between metric spaces that roughly preserves distances, up to constant multiplicative and additive terms; see~\cite{hickingbotham2025twquasiisom,coarsetw2025} for a formal definition. Quasi-isometries serve as a basic notion of equivalence in coarse graph theory.} to graphs of bounded treewidth (with parameters of the quasi-isometry depending only on $k$ and $r$). These works extend the earlier result of Berger and Seymour~\cite{BergerS24} that admitting a tree decomposition where every bag has bounded diameter, i.e., is $(1,r)$-coverable for a constant $r$, is equivalent to being quasi-isometric to a~tree.

For the second side of the equivalence --- balanced separators ---
we need a few definitions. Suppose $G$ is a graph and $\mu\colon V(G)\to \R_{\geq 0}$ is a weight function that assigns each vertex of $G$ a nonnegative weight. We say that a set $X$ of vertices of $G$ is a {\em{balanced separator}} for $\mu$ if for every component $C$ of $G-X$, the total weight of vertices within $C$ is at most half of the total weight of $G$. On the one hand, it is not hard to see that if $G$ admits a tree decomposition $\Tt$ of width at most $k$, then there is a bag of $\Tt$ that is a balanced separator for $\mu$; hence any weight function $\mu$ admits a balanced separator of size at most $k+1$. On the other hand, using standard approaches to approximating treewidth (see e.g.~\cite[Section~7.6]{platypus}) one can argue that if every weight function $\mu$ on a graph $G$ admits a balanced separator of size at most $\ell$, then the treewidth of $G$ is at most $3\ell$. Thus, treewidth and the {\em{balanced separator number}}, i.e., the smallest $\ell$ such that every weight function admits a balanced separator of size at most $\ell$, are bounded by linear functions of each other; see \cref{lem:bsn}.

Again, a natural coarse analogue of the balanced separator number would be to require that every vertex weight function admits a balanced separator that can be covered by a bounded number of bounded-radius balls. 
More precisely, we shall say that a graph $G$ has {\em{distance-$r$ balanced separator number}} at most $k$ if every weight function $\mu\colon V(G)\to \R_{\geq 0}$ admits a $(k,r)$-coverable balanced separator.

With the coarse analogues of both sides of the equivalence understood, we can state the main result of this paper: the coarse analogue of the equivalence between treewidth and the balanced separator number.

\begin{restatable}{theorem}{thmdecomp}
\label{thm:tdecomp}
    There is a constant $c\in \N$ such that the following holds:
    For all $k,r\in \N$ with $k\geq 2$, every graph whose distance-$r$ balanced separator number is at most $k$
    has a tree decomposition whose every bag is $(c k^2\log k, crk^2\log k)$-coverable.
\end{restatable}

Note that \cref{thm:tdecomp} gives only one direction of the equivalence, but the converse direction is easy: if $G$ admits a tree decomposition $\Tt$ whose bags are $(\ell,d)$-coverable and $\mu$ is a weight function on $V(G)$, then again there is a bag of $\Tt$ that is a balanced separator for $\mu$. Hence $\mu$ has an $(\ell,d)$-coverable balanced separator.

At a high level, in the proof of \cref{thm:tdecomp} we emulate the recursive algorithm from the standard proof of the connection between treewidth and balanced separators. In this algorithm, the part of the graph that is left to decompose is separated from the rest of the graph by a separator $S$. Classically, $S$ is bounded in size, but in our recursion we keep the invariant that $S$ can be covered by $\Oh(k^2\log k)$ balls. For technical reasons we need to allow the radii of those balls to grow in consecutive calls. To control the growth of the balls, we apply careful bookkeeping using a potential function that grows exponentially with the radii of the balls in the cover. We argue that the potential always stays bounded by a quantity expressed only in terms of $k$ and $r$; this also implies an upper bound on the maximum radius to which the balls may grow.

\paragraph*{Relation to the previous version and AI disclosure.} In the previous version of this manuscript, the statement of \cref{thm:tdecomp} was left open as a conjecture: we postulated the bags to be $(\ell,d)$-coverable for some $\ell$ and $d$ bounded in terms of $k$ and $r$. We provided a proof of a relaxation of this conjecture that achieved $(ck^2\log k,r(\log k+\log \log n+c))$-coverability of the bags, where $n$ is the number of vertices of the graph. This proof followed exactly the same approach and used exactly the same potential function as the proof presented in this version, except that it lacked the argument provided by \cref{lem:binary_gap} that allows one to bound the potential by a function of $k$ and $r$ only. Instead, the potential was allowed to grow logarithmically with the vertex count $n$, resulting in a doubly logarithmic bound on the final radii. The missing argument --- \cref{lem:binary_gap} and its application in the potential accounting scheme --- was proposed by ChatGPT 6.0 Astra. We verified the argument and adopted it in the write-up. We would like to thank Bart\l{}omiej Lewandowski for his help in prompting.
We are also grateful to J\'ulia Balig\'acs, Václav Bla\v{z}ej, 
Evangelos Protopapas, and Amadeus Reinald for inspiring discussions on the topic.

The previous version of this manuscript also contained a proof of the conjecture in the case of graphs with bounded doubling dimension. This proof has been removed because the developments described above made it obsolete.

All other ideas, as well as
the entire presentation, are due to the authors. Claude Opus 5.5 was used for proofreading and minor editing.

\paragraph*{Further research.} It remains open whether in \cref{thm:tdecomp}, the radii of the balls covering the bags must also depend on $k$; in other words, whether one can obtain a variant with $(\ell,d)$-coverable bags, where $\ell$ is bounded by a function of $k$ and $r$ while $d$ is bounded by a function of $r$ only. As proved by Chudnovsky and Hickingbotham~\cite{ChudnovskyH25}, this is the case when $r=1$ and the graph excludes some fixed biclique as an induced subgraph (then $d=2$ suffices).
\section{Preliminaries}\label{sec:prelim}

By $\N$ and $\R_{\geq 0}$ we denote the sets of nonnegative integers and nonnegative reals, respectively. All logarithms are base-$2$.

\paragraph*{Graphs.} We use standard graph notation. All graphs considered in this paper are finite, simple, and unweighted.
The set of vertices of a graph $G$ is denoted by $V(G)$ and the set of edges by $E(G)$. 

The distance metric of a graph $G$ is denoted by $\dist_G(\cdot,\cdot)$: for two vertices $u,v\in V(G)$, their distance, denoted by $\dist_G(u,v)$, is the length (i.e., the number of edges) of a shortest path connecting $u$ and $v$, or $+\infty$ if no such path exists. For a vertex $u$ and a nonnegative real~$r$, the {\em{radius-$r$ ball}} around $u$ is the set $\B_G(u,r)\coloneqq \{v\in V(G)\mid \dist_G(u,v)\leq r\}$. We say that a set $A$ of vertices in a graph $G$  is {\em{$(k,r)$-coverable}} if one can choose at most $k$ radius-$r$ balls in $G$ whose union contains $A$. We emphasize that the distances here are measured in $G$ and the balls do not have to be contained in $A$.

In all of the notation above, we may omit the graph in the subscript if it is clear from the context.

\paragraph*{Tree decompositions.} A \emph{tree decomposition} of a graph $G$ is a pair $\mathcal{T}=(T,\bag)$, where $T$ is a tree and $\bag$ is a function assigning to every node $t$ of $T$ a subset $\bag(t)$  of vertices of $G$ (called the {\it bag} of $t$) such that the following hold:
\begin{itemize}[nosep]
    \item for each edge $vu\in E(G)$, there is a node $t$ of $T$ such that $\bag(t)$ contains both $u$ and $v$; and 
    \item for each vertex $u$ of $G$, the set of nodes of $T$ whose bags contain $u$  induces a connected nonempty subtree of $T$.
\end{itemize}
The \emph{width} of a tree decomposition is equal to $\max_{t\in V(T)}|\bag(t)|-1$. The minimum width over all tree decompositions of a graph $G$ is called the \emph{treewidth} of $G$ and is denoted by $\tw(G)$.

\paragraph*{Balanced separators.} Let $G$ be a graph and $\mu\colon V(G)\to \R_{\geq 0}$ be a weight function on the vertices of $G$. For a subset of vertices $A$ we denote $\mu(A)\coloneqq \sum_{u\in A} \mu(u)$.

We say that a vertex set $X\subseteq V(G)$ is a {\em{balanced separator}} for $\mu$ if for every component $C$ of $G-X$, we have $\mu(V(C))\leq \frac{1}{2}\mu(V(G))$. The {\em{balanced separator number}} of $G$, denoted $\bsn(G)$, is the smallest $k$ such that for every weight function $\mu$ there is a balanced separator of size at most $k$. We have the following standard lemma that connects the treewidth with the balanced separator number.

\begin{lemma}[{see e.g.~\cite[Section~5]{HarveyW17}}]\label{lem:bsn}
 For any graph $G$, we have
 \[\bsn(G)-1 \leq \tw(G)\leq 3\bsn(G).\]
\end{lemma}

As mentioned in \cref{sec:intro}, we will work with the following coarse variant of the balanced separator number. For $r\in \R_{\geq 0}$, the {\em{distance-$r$ balanced separator number}} of $G$ is the smallest $k$ such that for every weight function $\mu\colon V(G)\to \R_{\geq 0}$, there is a balanced separator for $\mu$ that is $(k,r)$-coverable.

\section{Main proof}

In order to prove \cref{thm:tdecomp}, we will construct a suitable tree decomposition explicitly, using a recursive procedure similar to the one used in the classic algorithms for constructing tree decompositions of graphs, see e.g.~\cite[Section~7.6]{platypus}. We recursively split the graph using balanced separators, compressing the accumulated covering balls to control both their number and their radii.

To facilitate the description of our recursive procedure, we need the following definition of a partial tree decomposition that encapsulates the task of decomposing a subgraph of the given graph. Let $G$ be a graph, $S$ be a subset of vertices of $G$, and $U$ be the vertex set of a component of $G - S$. Then a \emph{partial tree decomposition} of $(S, U)$ is a tree decomposition $\Tt$ of $G[S\cup U]$ with the following additional property: there is a bag of $\Tt$ that contains the whole $S$.

We will need a few extra definitions. Recall that in the context of \cref{thm:tdecomp}, we assume the existence of balanced separators consisting of $k$ balls of radius~$r$, but to cover the bags of the constructed decomposition, we allow balls of varying radii. We will only use balls whose radii are multiples of $r$, so we call a set $\Bb$ of balls {\em{round}} if for every ball $B\in \Bb$, the radius of $B$, denoted $\rad(B)$, is a positive integer multiple of $r$. We will use the following {\em{potential}} of $\Bb$ to keep track of the growth of radii: \[\Phi(\Bb) \coloneqq \sum_{B \in \Bb} 2^{\rad(B)/r}.\]
Note that these definitions depend on the radius parameter $r$, a positive integer fixed in the context.

Also, we set
\[\alpha \coloneqq 2 + \ceil{\log 2k}\quad \text{and}\quad \Gamma\coloneqq 8\alpha k\cdot 2^\alpha+4k.\]
The value $\Gamma$ will be the bound on the number of balls needed to cover every separator of the constructed tree decomposition. Again, these definitions depend on the parameter $k$, a positive integer fixed in the context.

In controlling the potential, we will also use the following lemma:
\begin{lemma}\label{lem:binary_gap}
    Let $m,h\in \N$. Suppose $p$ is an integer that can be written as a sum of at most $m$ powers of $2$ and satisfies $p < 2^{m+h}$. Then also $p \leq 2^{m+h}- 2^h$.
\end{lemma}
\begin{proof}
    By exhaustively replacing any two equal powers of $2$ with their sum, we may assume that $p$ is the sum of at most $m$ pairwise different powers of $2$. The $m$ largest powers of $2$ smaller than $2^{m+h}$ are $2^{m+h-1},2^{m+h-2},\ldots,2^h$, hence we have
    \[
        p \leq 2^{m+h-1}+2^{m+h-2}+\ldots+2^h = 2^{m+h} - 2^h.
    \]
    This completes the proof of the lemma.
\end{proof}

With these definitions in place, our recursive procedure can be captured by the lemma below.

\begin{lemma}\label{lem:tdecomp_step}
    Fix a positive integer $r$.
    Let $G$ be a graph with distance-$r$ balanced separator number at most $k$, let $S$ be a set of vertices of $G$, and let $U$ be the vertex set of a component of $G - S$. Let $\Bb$ be a round set of at most $\Gamma$ balls such that $S \subseteq \bigcup \Bb$ and $\Phi(\mathcal{B}) < 2^{\Gamma + \alpha}$. Then there exists a partial tree decomposition $\Tt$ of $(S, U)$ such that each bag of $\Tt$
    can be covered by a round set of at most $\Gamma+2k$ balls with potential at most $2^{\Gamma + \alpha}$.
\end{lemma}
\begin{proof}
    The proof is by induction on the size of $U$, which is positive. If $|U|=1$, say $U=\{u\}$, we take a tree decomposition consisting of a single bag $S\cup\{u\}$. This bag is covered by the round set $\Bb\cup\{\B_G(u,r)\}$, which has at most $\Gamma+1\leq\Gamma+2k$ balls and potential at most $\Phi(\Bb)+2\leq 2^{\Gamma+\alpha}$. The last inequality follows because $\Phi(\Bb)$ and $2^{\Gamma+\alpha}$ are even integers and $\Phi(\Bb)<2^{\Gamma+\alpha}$.

    We proceed to the induction step, assuming $|U|\geq 2$. First, we need to process the ball set $\Bb$ in order to achieve a certain ``sparseness'' property that will be useful later.

    For a positive integer $\ell$, we shall say that a vertex $x \in V(G)$  is {\em{$\ell$-crowded}} with respect to $\Bb$ if there exist at least $2^{\alpha}$ balls in $\Bb$ of radius exactly $\ell r$ whose centers are at distance at most $\alpha r$ from $x$. Let $\Bb'$ denote the set obtained from $\Bb$ by replacing all such balls with a single ball of radius $r(\ell + \alpha)$ with center at~$x$. Clearly, we have
    \[|\Bb'|<|\Bb|\qquad\textrm{and}\qquad \Phi(\Bb')\leq \Phi(\Bb)-2^\alpha\cdot 2^\ell+2^{\ell+\alpha}=\Phi(\Bb).\]
    Moreover, by the triangle inequality, every ball removed from $\Bb$ is entirely covered by the ball added, hence
    \[\bigcup \Bb'\supseteq \bigcup \Bb.\]
    By applying this operation repeatedly as long as there exists an $\ell$-crowded vertex for some $\ell$, we arrive at a round set of balls $\Bb'$ such that 
    \begin{itemize}[nosep]
        \item $|\Bb'| \leq |\Bb|$,
        \item $\Phi(\Bb') \leq \Phi(\Bb)$,
        \item $\bigcup \Bb'\supseteq \bigcup \Bb$, and
        \item no vertex $u\in V(G)$ is $\ell$-crowded with respect to $\Bb'$, for any positive integer $\ell$.
    \end{itemize}
    Also, without loss of generality we may assume that no ball in $\Bb'$ is entirely contained in another ball from $\Bb'$, for the smaller ball can simply be removed from $\Bb'$ without breaking any of the properties above.

    We observe the following consequence of the construction of $\Bb'$.

    \begin{claim}\label{clm:balls_no_bound}
        For every vertex $u \in V(G)$, there are at most $2\alpha \cdot 2^{\alpha}$ balls in $\Bb'$ whose centers lie at distance at most $\alpha r$ from $u$.
    \end{claim}
    \begin{claimproof}
        Let $\Bb'_u$ be the set of balls from $\Bb'$ whose centers are at distance at most $\alpha r$ from $u$. Pick any two distinct balls $B_1, B_2 \in \Bb'_u$, say with centers $o_1, o_2$  and radii $\ell_1 r, \ell_2 r$, respectively. Suppose for a moment that $\ell_1 \geq \ell_2 + 2 \alpha$. Then for every vertex $v \in B_2$, we have
        \[
            \dist(v, o_1) \leq \dist(v, o_2) + \dist(o_2, u) + \dist(u, o_1) \leq
            (\ell_2 + 2 \alpha) r \leq \ell_1 r,
        \]
        implying $B_2 \subseteq B_1$, which is a contradiction with the construction of $\Bb'$. Similarly, the supposition $\ell_2\geq \ell_1+2\alpha$ also leads to a contradiction. Therefore we have $|\ell_1 - \ell_2| < 2 \alpha$. Since $B_1,B_2$ were chosen arbitrarily, we conclude that there are at most $2\alpha$ different radii among the balls of $\Bb'_u$. Since $u$ is not $\ell$-crowded with respect to $\Bb'$ for any positive integer $\ell$, every fixed radius gives rise to at most $2^{\alpha}$ balls of $\Bb'_u$. We conclude that $|\Bb'_u| \leq 2\alpha \cdot 2^{\alpha}$, as claimed.
    \end{claimproof}

    With $\Bb'$ constructed, we proceed with the proof.
    Since the balls of $\Bb'$ are not contained one in another, they have pairwise different centers. Let $O$ be the set of centers of balls from $\Bb'$.
    Let $\mu_U,\mu_O\colon V(G)\to\{0,1\}$ be the indicator functions of $U$ and $O$, respectively, i.e.,
    \[
        \mu_U(u)\coloneqq
        \begin{cases}
            1 & \text{if }u\in U,\\
            0 & \text{otherwise}
        \end{cases}
        \qquad\text{and}\qquad
        \mu_O(u)\coloneqq
        \begin{cases}
            1 & \text{if }u\in O,\\
            0 & \text{otherwise.}
        \end{cases}
    \]
    Since $G$ has distance-$r$ balanced separator number bounded by $k$, we may find sets  $\Dd_U$ and $\Dd_O$ of radius-$r$ balls with $|\Dd_U|,|\Dd_O|\leq k$ such that $\bigcup \Dd_U$ is a balanced separator for $\mu_U$ and $\bigcup \Dd_O$ is a balanced separator for $\mu_O$. Define
    \[\Dd\coloneqq \Dd_U\cup \Dd_O\qquad\textrm{and}\qquad Z\coloneqq \bigcup \Dd.\]
    Observe that $|\Dd|\leq 2k$ and $Z$ is a balanced separator both for $\mu_U$ and for $\mu_O$. 
    In what follows we will define a number of sets of balls, see \cref{fig:balls}.

    Let $O_\Dd$ be the set of centers of balls from $\Dd$. Further, let $\wh{\Dd}$ be the set of those balls from $\Bb'$ whose centers lie at distance at most $\alpha r$ from some vertex of~$O_\Dd$. By \cref{clm:balls_no_bound}, we have \[|\wh{\Dd}| \leq 2\alpha \cdot 2^{\alpha} \cdot |O_\Dd| \leq 4\alpha k \cdot 2^{\alpha}.\]

    Let $W$ be the vertex set of any component of $G-Z$. Since $Z$ is a balanced separator for $\mu_U$ and for $\mu_O$, we have \[|W \cap U| \leq |U|/2 \qquad\textrm{and}\qquad |W \cap O| \leq |O|/2\leq \Gamma / 2.\] Consider all balls of $\Bb' - \wh{\Dd}$ with centers outside of $W$ and let $R_W$ denote the largest radius among them. We set $R_W \coloneqq (\alpha - 1)r$ in case there are no such balls or all such balls have radius less than $(\alpha - 1)r$. Let $\Dd_W$ be the set of balls obtained from the set $\Dd$ by changing the radius of every ball to $R_W - (\alpha - 2)r$ (note that this value is at least $r$). 
    Next,
let $\Bb_W'$ be those balls from $\Bb'$ whose centers lie in $W$. We~set \[\Bb_W \coloneqq \wh{\Dd} \cup \Dd_W \cup \Bb'_W.\]

\begin{figure}[t]
    \centering
    \includegraphics[width=0.7\linewidth]{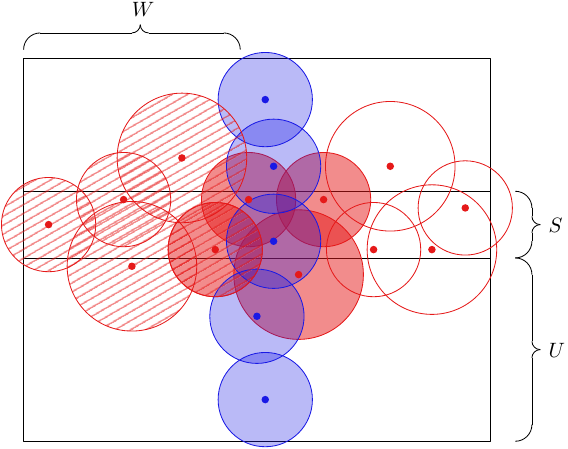}
    \caption{Various sets of balls in the proof of \cref{lem:tdecomp_step}. Blue disks indicate balls from $\Dd$, and their centers, i.e., the set $O_{\Dd}$, are shown by blue dots.
    Balls  $\Dd_W$ are obtained by (possibly) enlarging balls from $\Dd$, while keeping the same centers; this is not shown in the picture for the sake of clarity.
    Red disks depict balls from $\Bb'$. Balls filled with diagonal lines have their centers in $W$, i.e., they belong to~$\Bb'_W$. Filled red disks indicate the set $\wh{\Dd}$, i.e., their centers are close to the vertices from $O_{\Dd}$.    
    }
    \label{fig:balls}
\end{figure}

 Note that
    \[
    |\mathcal{B}_W| \leq |\wh{\Dd}| + |\Dd_W| + |\mathcal{B}'_W| \leq
    4\alpha k \cdot 2^{\alpha} + 2k + \Gamma/2 = \Gamma,
    \]
    where the equality follows from the definition of $\Gamma$.

Now let us bound the potential of $\Bb_W$.
We consider two cases.
Suppose first that there is no ball $B \in \Bb' - \wh{\Dd}$ with center outside of $W$ and radius at least $(\alpha - 1) r$.
Then $\Dd_W=\Dd$ and all the balls of $\Dd_W$ have radius $r$, hence 
    \[
    \Phi(\mathcal{B}_W) \leq \Phi(\wh{\Dd} \cup \mathcal{B}'_W) + \Phi(\Dd_W) \leq \Phi(\Bb') + |\Dd_W|\cdot 2\leq \Phi(\Bb) + 4k.
    \]
So now suppose that such a ball $B$ exists, and pick one with maximum radius. Since $B \notin \wh{\Dd} \cup \Bb'_W$, we have
    \[
    \Phi(\Bb_W) \leq \Phi(\wh{\Dd} \cup \Bb'_W) + \Phi(\Dd_W) \leq
    \Phi(\Bb' - \{B\}) + \Phi(\Dd_W) \leq
    \Phi(\Bb') - 2^{R_W/r} + 2k \cdot 2^{R_W/r - (\alpha - 2)}.
    \]
    We have $\alpha - 2 = \ceil{\log 2k}$, so
    \[
    \Phi(\Bb_W) \leq \Phi(\Bb') - 2^{R_W/r} + 2^{R_W/r} \cdot (2k \cdot 2^{-\log 2k}) = \Phi(\Bb')\leq \Phi(\Bb).
    \]

    The intuition of the remainder of the proof is as follows. Recall that $G[U]$ is a component of $G-S$, hence the removal of $Z$ breaks $G[U]$ into several smaller components. Each such component $G[A]$ is contained in some component of $G-Z$, say $G[W]$. We would like to apply induction to all components $G[A]$ as above, keeping $\Bb_W$ as the cover for a suitable set $S_A$ separating $A$ from the rest of the graph. However, the construction of $\Bb_W$ used only a subset of the balls from $\Bb'$, so we need to make sure that the part of $S$ contained in $S_A$ is still covered by $\Bb_W$.
    We do this in the following claim.

    \begin{claim}\label{clm:no_leaks}
        Let $W$ be the vertex set of a component of $G-Z$. Then every ball $B \in \Bb' - \wh{\Dd}$ whose center lies outside of $W$ is disjoint from $W - \bigcup \Dd_W$.
    \end{claim}

    \begin{claimproof}
        Pick any ball $B\in \Bb'-\wh{\Dd}$, say of radius $r' \leq R_W$, and let $o$ be its center; assume $o\notin W$. Pick any vertex $x \in W$ with $\dist(x, o) \leq r'$ and let $P$ be a shortest path connecting $o$ and $x$. Since $o \notin W$ and $x\in W$, there exists a ball $B' \in \Dd$ which intersects $P$; recall that $B'$ has radius $r$.
        Let $o'$ denote the center of $B'$ and let $z$ be any vertex on $P$ that belongs to $B'$.
        As $P$ is a shortest path, we have that $\dist(o,x) = \dist(o,z) + \dist(z,x)$, and since $z \in B'$, we have $\dist(o',z)\leq r$.
        Since $B \notin \wh{\Dd}$, we have
        $
        \dist(o, z) \geq \dist(o, o') - \dist(o', z) \geq (\alpha - 1) r,
        $
        hence
        \begin{align*}
        \dist(x, o') \leq \ &  \dist(o', z) + \dist(z, x) = \dist(o', z) + \dist(o, x) - \dist(o,z)\\
        \leq \ &  r + r' - \dist(o, z) \leq  r + R_W - (\alpha - 1)r = R_W - (\alpha - 2)r.
        \end{align*}
        In particular, the ball of radius $R_W - (\alpha - 2)r$ with center at $o'$ both contains $x$ and belongs to $\Dd_W$. So $x\in \bigcup \Dd_W$, and the proof is complete.
    \end{claimproof}

    Let $\Aa$ comprise the vertex sets of all the components of $G[U]-Z$.
    For $A\in \Aa$, let $G[W_A]$ be the component of $G-Z$ such that $A\subseteq W_A$. We define
    \[S_A \coloneqq (Z \cap (U \cup S)) \cup (S \cap W_A).\]
    By construction, $G[A]$ is connected, disjoint from $S_A$, and $N_G(A)\subseteq S_A$, so it is a component of $G-S_A$.
    Let us first verify that the ball set $\Bb_{W_A}$ covers $S_A$.
    
    \begin{claim}\label{cl:coverage}
        For every $A\in \Aa$, we have $S_A\subseteq \bigcup \Bb_{W_A}$.
    \end{claim}

    \begin{claimproof}
        Pick any $x \in S_A$. If $x \in Z$, then $x$ is covered by some ball in $\Dd_{W_A}\subseteq \Bb_{W_A}$ (recall that every ball of $\Dd_{W_A}$ is obtained from a ball of $\Dd$ by possibly increasing the radius). 
        Now, assume that $x \in S \cap W_A$, so in particular there exists a ball $B \in \Bb'$ that contains $x$. If the center of $B$ lies in $W_A$, then $B\in \Bb'_{W_A}\subseteq \Bb_{W_A}$ and consequently $x\in \bigcup \Bb_{W_A}$. If $B\in \wh{\Dd}$, then $x\in \bigcup \wh{\Dd}\subseteq \bigcup \Bb_{W_A}$. And if $B\in \Bb'-\wh{\Dd}$ and the center of $B$ lies outside of $W_A$, then by \cref{clm:no_leaks} we have $x\in \bigcup \Dd_{W_A}\subseteq \bigcup \Bb_{W_A}$.
    \end{claimproof}

    Observe that since $|W\cap U|\leq |U|/2$ for each component $W$ of $G-Z$, we also have $|A|\leq |U|/2$ for each $A\in \Aa$. Also, by \cref{cl:coverage} we have that $\Bb_{W_A}$ covers $S_A$, and recall that $|\Bb_{W_A}|\leq \Gamma$ and $\Phi(\Bb_{W_A})\leq \Phi(\Bb)+4k$.
    
    We know that $\Phi(\mathcal{B})$ is a sum of at most $\Gamma$ powers of $2$ and $\Phi(\mathcal{B})< 2^{\Gamma +\alpha}$, so by \Cref{lem:binary_gap} we have 
    \begin{equation}\label{eq:magic}
        \Phi(\mathcal{B}) \leq 2^{\Gamma + \alpha} - 2^\alpha.
    \end{equation}
    Moreover,
    \begin{equation}\label{eq:truth_about_k}
        4k = 2^{2 + \log k} < 2^{2 + \log 2k} \leq 2^\alpha.
    \end{equation}
    Thus,
    \[\Phi(\Bb_{W_A})\leq \Phi(\Bb)+4k < 2^{\Gamma + \alpha} - 2^\alpha + 2^\alpha = 2^{\Gamma + \alpha}. \]

    Therefore, we may apply induction to the pair $(S_A,A)$ for each $A\in \Aa$, thus obtaining a partial tree decomposition $\Tt_A$ of $(S_A,A)$ with the following properties:
    \begin{itemize}[nosep]
        \item $\Tt_A$ has a node, say $t_A$, whose bag contains the whole $S_A$; and
        \item every bag of $\Tt_A$ can be covered by a round set of at most $\Gamma+2k$ balls with potential bounded by $2^{\Gamma +\alpha}$.
    \end{itemize}
    We now combine the decompositions $\{\Tt_A \mid A\in \Aa\}$ into a tree decomposition $\Tt$ of $G[S\cup U]$ by creating a fresh node $t$ with bag $S\cup (Z\cap U)$ and making it adjacent to all the nodes $t_A$ for $A\in \Aa$. It is straightforward to verify that $\Tt$ is a tree decomposition of $G[S\cup U]$; we leave the verification to the reader. Since $S$ is contained in the bag of $t$, $\Tt$ is a partial  tree decomposition of $(S,U)$.

    It only remains to check whether the bag of $t$ --- namely $S\cup (Z\cap U)$ --- can be covered by a round set of at most $\Gamma+2k$ balls with potential bounded by $2^{\Gamma + \alpha}$. For this, we take $\Bb\cup \Dd$. Note that
    \[\bigcup (\Bb\cup \Dd)=\bigcup \Bb\cup \bigcup \Dd\supseteq S\cup Z\supseteq S\cup (Z\cap U).\]
    Finally, we have
    \[|\Bb\cup \Dd|\leq |\Bb|+|\Dd|\leq \Gamma+2k\]
    \[\textrm{and}\quad \Phi(\Bb\cup \Dd)\leq \Phi(\Bb)+\Phi(\Dd)\leq \Phi(\Bb)+4k < 2^{\Gamma +\alpha} - 2^\alpha + 2^\alpha = 2^{\Gamma + \alpha},\]
    where we use that each ball of $\Dd$ has potential $2$, as well as \eqref{eq:magic} and \eqref{eq:truth_about_k}.
\end{proof}

Now \cref{thm:tdecomp} follows from an easy application of \cref{lem:tdecomp_step}.

\thmdecomp*
\begin{proof}
    Note that the case $r=0$ is handled by \cref{lem:bsn}, so we may assume that $r \geq 1$.
    Assuming without loss of generality that $G$ is connected, we apply \cref{lem:tdecomp_step} to $S=\emptyset$, $U=V(G)$, and the empty set of balls. Thus, we obtain a tree decomposition $\Tt$ of $G$ whose every bag can be covered by a round set of $K \leq \Gamma+2k = 8\alpha k\cdot 2^\alpha+6k$ balls whose potential is at most $2^{\Gamma+\alpha}$. Let $\Bb$ be any such set and denote $R \coloneqq \max_{B \in \Bb} \rad(B)$. We have
    \[
    2^{R/r} \leq \Phi(\mathcal{B}) \leq2^{\Gamma+\alpha}.
    \]
    Taking a logarithm, we get $R \leq r(\Gamma+\alpha)$.

    Let us now state these bounds as functions of $k$ and $r$:
    \begin{align*}
        K & \leq \Gamma+2k=8\alpha k\cdot 2^\alpha+6k=\mathcal{O}(k^2\log k)\\
        R & \leq r(\Gamma+\alpha) = \mathcal{O}(rk^2\log k).
    \end{align*}
    Therefore each bag can be covered by $K=\mathcal{O}(k^2\log k)$ balls of radius at most $R=\mathcal{O}(rk^2\log k)$. Picking $c$ large enough, we obtain the statement of the theorem.
\end{proof}

\bibliographystyle{plain}
\bibliography{bibliography}

\end{document}